\documentclass{amsart}

\usepackage{amsmath,amssymb,amsfonts,graphics,color}
\usepackage{epsfig}
\usepackage{latexsym}
\usepackage{euscript}
\newcommand{\ignore}[1]{}

\title{Amenability properties of Rajchman algebras}

\author[Mahya Ghandehari]{Mahya Ghandehari}
\address{Department of Mathematics and Statistics, Dalhousie University, Halifax, Nova Scotia, Canada, B3H 3J5}

\thanks{}

\email{ghandeh@mathstat.dal.ca}

\date{\today}

\newtheorem{theorem}{Theorem}[section]
\newtheorem{corollary}[theorem]{Corollary}

\newtheorem{lemma}[theorem]{Lemma}
\newtheorem{proposition}[theorem]{Proposition}

\newcommand{\cal}{\mathcal}

\def\sup{{\rm sup}}

\def\VN {{\rm VN}}
\def\Ind {{\rm Ind}}
\def\SL {{\rm SL}_2({\mathbb R})}
\def\GL {{\rm GL}_n({\mathbb R})}
\def\Hn {{\rm H}_n({\mathbb R})}

\def\Ind{{\rm Ind}}

\def\supp {{\rm supp}}

\def\Span {{\rm Span}}

\def\Tr {{\rm Tr}}

\begin{document}
\maketitle

\begin{abstract}
Rajchman measures of locally compact Abelian groups are studied for almost a century now, and they play an important role in the study of trigonometric series. Eymard's influential work allowed generalizing these measures to the case of \emph{non-Abelian} locally compact groups $G$. The Rajchman algebra of $G$, which we denote by $B_0(G)$, is the set of all elements of the Fourier-Stieltjes algebra that vanish at infinity.

In the present article, we  characterize the locally compact groups that have amenable Rajchman algebras. We show that $B_0(G)$ is amenable if and only if $G$ is compact and almost Abelian.
On the other extreme, we present many examples of locally compact groups, such as non-compact Abelian groups and infinite solvable groups, for which $B_0(G)$ fails to even have an approximate identity.
\end{abstract}

\newline
{{\sc Keywords:} Rajchman algebra, amenability, operator amenability, bounded approximate identity.

\section{Introduction \label{sec:intro}}

Amenability of a group is a fundamental notion in  analysis that was originally introduced by  von Neumann in 1929.
This remarkable property has many equivalent definitions and various
interpretations. For instance, one can think of amenability as the existence of a
translation-invariant averaging condition for a locally compact group.
In 1972, Johnson defined amenable Banach algebras as those satisfying a
certain cohomological property \cite{johnsonM}. The choice of terminology was
inspired by Johnson's well-known theorem demonstrating the equivalence of amenability for a locally compact group and its
convolution algebra \cite{johnsonM}. The concept of amenability turned out to be extremely
fruitful in the theory of  Banach algebras. For example,  Connes~\cite{connes78} and  Haagerup~\cite{haagerup83}  showed that for $C^*$-algebras, amenability and nuclearity coincide.

We recall that a measure $\mu$ in the measure algebra of a locally compact Abelian group is called a Rajchman measure if
$$\lim_{\chi\rightarrow \infty}\hat{\mu}(\chi)=0.$$
The importance of Rajchman measures first became apparent in the study of uniqueness of trigonometric series.  In 1916, Menshov  showed that there are closed sets of Lebesgue measure zero which are not sets of uniqueness~\cite{Menshov}. In his proof, Menshov constructs a probability measure $\mu$ supported in a set of Lebesgue measure zero whose Fourier transform vanishes at infinity. This is one of the earliest examples of measures in $M_0({\mathbb T})$ which do not belong to $L^1({\mathbb T})$. Hewitt and Zuckerman  generalized this result for all non-discrete locally compact Abelian groups~\cite{Hewitt-Zuckerman}.

In his influential work, Eymard \cite{eymard} defined the Fourier and Fourier-Stieltjes algebras of locally compact groups, and studied many of their properties. This in particular gave rise to a natural generalization of Rajchman measures.  The \emph{Rajchman algebra} associated with a locally compact group $G$, denoted by $B_0(G)$, is the set of elements of the Fourier-Stieltjes algebra which vanish at infinity.
It is easy to see that the Rajchman algebra is indeed a Banach subalgebra of the Fourier-Stieltjes algebra.
Note that the Rajchman algebra of a locally compact Abelian group can be identified with the algebra of Rajchman measures on the dual group,
denoted by $M_0(\hat{G})$.

The Rajchman algebra of a locally compact group $G$ is the smallest algebra that contains the coefficient space associated with every $C_0$-representation of $G$. The $C_0$-representations of a locally compact group have been studied in various papers such as \cite{Taylor-vanishing-rep} and \cite{Howe-Moore}.
Understanding such unitary representations and their coefficient spaces, which are certain subspaces of $B_0(G)$, is important
due to their applications in other areas of mathematics such as the theory of automorphic forms and ergodic properties of flows on homogeneous spaces (e.g. see \cite{Moore-flow} and \cite{Shintani}).

The purpose of this article is to investigate $B_0(G)$ as a Banach algebra.
We obtain a characterization of the locally compact groups that have amenable Rajchman algebras.
Namely, we prove that amenability of the Rajchman algebra of a group  is equivalent to the group being compact and almost Abelian, which in turn is
equivalent to the amenability of its Fourier-Stieltjes algebra.
On the other hand, we present examples of groups, such as non-compact Abelian groups and infinite solvable groups, for which Rajchman algebras do not even have a bounded approximate identity. This shows that Rajchman algebras behave widely in terms of amenability, since the existence of a bounded approximate identity
is a necessary condition for amenability or the so-called operator amenability of an algebra.

Many important Banach algebras in harmonic
analysis, e.g. the Fourier and the Fourier-Stieltjes algebras,  are operator spaces as well. Thus, it is natural to also define the notion of
operator amenability in order to take the operator space structure
into account.
%
Combining the famous theorems of Johnson~\cite{johnsonM} and  Ruan~\cite{ruan}, one observes that
for a locally compact group, the  amenability of the
$L^1$-algebra and the operator  amenability of the Fourier
algebra are equivalent. This fact leads one to suspect the analogous relation between the measure
algebra and the Fourier-Stieltjes algebra.
For a locally compact group, it has been
shown that the measure algebra is amenable if and only if the group is discrete and amenable~\cite{DGH}. Since compactness is the dual notion to
discreteness, it is natural to conjecture that the operator amenability of the Fourier-Stieltjes algebra is equivalent to the compactness of the group. In 2007,  Runde and Spronk~\cite{rundes2} found surprising examples of
noncompact operator amenable Fell groups. These examples disproved
the conjecture, and left the characterization of operator
amenability of the Fourier-Stieltjes algebras wide open.
Our investigation leads us to believe that Rajchman algebras of many locally compact groups seem to have as rich a structure as their Fourier-Stieltjes algebras, and can be used as a crucial stepping stone in the study of  the operator amenability of the Fourier-Stieltjes algebras.

\paragraph{}
The remainder of this article is organized as follows.
In Section~\ref{section-background}, we provide the necessary background, review some basics of noncommutative harmonic analysis, and briefly discuss  induced representations.
In Section \ref{section:abelian}, we study $M_0(G)$ for non-discrete locally compact Abelian groups. Among other
results, we show that the Rajchman algebra of a non-discrete locally compact Abelian group does not have an approximate identity, which implies that it is not (operator) amenable. In Section \ref{section:functorial}, we discuss the functorial properties of Rajchman algebras in certain cases such as SIN-groups.
 A locally compact group is called a SIN-group if it has a neighborhood basis of the identity consisting of pre-compact neighborhoods which are invariant under inner automorphisms. This is a very natural class of groups which contains all Abelian, all compact and all discrete groups.
These functorial properties allow us to extend our non-amenability results to all non-compact connected SIN-groups.
The results in Section \ref{section:functorial} are used in Section \ref{section:amenability-of-B0} to establish the sufficient and necessary conditions for the amenability of Rajchman algebras.
In Section \ref{section:amenability-of-B0}, we also present examples of large classes of locally compact groups, such as non-compact connected SIN-groups and infinite solvable groups,
for which Rajchman algebras are not operator amenable.

\section*{Acknowledgements}
The research presented in this paper was part of the authors Ph.D. thesis. I would like to thank my supervisors Brian Forrest and Nico Spronk for their encouragement and invaluable discussions and suggestions. Many thanks to Viktor Losert for bringing Theorem \ref{b0-SL-not-sq-dense} to my attention and to Ebrahim Samei for his suggestions that led to Theorem \ref{thm:solvable}.

\section{Preliminaries}\label{section-background}
Let $G$ be a locally compact group with the Haar measure $\lambda_G$. Let the group algebra of $G$, denoted by $L^1(G)$, be the Lebesgue space
$L^1(G,\lambda_G)$. Recall that $L^1(G)$ equipped with pointwise addition and convolution is a Banach algebra.
Let $M(G)$ be the space of complex-valued Radon measures on $G$. We define  the convolution of two measures $\mu$ and $\nu$ in $M(G)$ to be
$$\int_Gf(z)d(\mu*\nu)(z)=\int_G\int_Gf(xy)d\mu(x)d\nu(y),$$
for every $f$ in $C_c(G)$, the set of compactly supported continuous functions on $G$. The measure algebra $M(G)$ equipped with the total variation norm is in fact a Banach algebra, which contains the group algebra as a closed ideal.

Let ${\cal H}$ be a Hilbert space, and ${\cal U}({\cal H})$ denote the group of unitary operators on ${\cal H}$.
A \emph{continuous unitary representation} of  $G$ on ${\cal H}$ is a group homomorphism $\pi:G\rightarrow {\cal U}({\cal H})$ which is WOT-continuous, i.e.
for every vector $\xi$ and $\eta$ in ${\cal H}$, the function
$$\xi*_\pi\eta:G\rightarrow {\mathbb C}, \quad g\mapsto \langle\pi(g)\xi,\eta\rangle$$
is continuous. Functions of the form $\xi*_\pi\eta$, for vectors $\xi$ and $\eta$ in ${\cal H}$, are called the \emph{coefficient functions}
of $G$ associated with the representation $\pi$.
One can extend $\pi$ to a non-degenerate norm-decreasing $*$-representation of the Banach $*$-algebra $L^1(G)$ to ${\cal B}({\cal H})$ via
$$\langle\pi(f)\xi,\eta\rangle=\int_G f(x)\langle\pi(x)\xi,\eta\rangle dx,$$
for every $f$ in $L^1(G)$ and vectors $\xi$ and $\eta$ in ${\cal H}$. We use the same symbol $\pi$ to denote the $*$-representation extension as well.

For a locally compact group $G$, the \emph{Fourier-Stieltjes algebra} of $G$, denoted by $B(G)$, is the set of all the coefficient functions of $G$ equipped
with pointwise algebra operations.
Eymard \cite{eymard} proved that $B(G)$ can be identified with the dual of the group $C^*$-algebra of $G$.
Moreover, the Fourier-Stieltjes algebra together with the norm from the above duality turns out to be a Banach algebra.
The \emph{Fourier algebra} of $G$, denoted by $A(G)$,  is  the closed subalgebra of the Fourier-Stieltjes algebra generated by its compactly supported elements. Clearly, the Fourier algebra is a subalgebra of $C_0(G)$, the algebra of continuous functions on $G$ which vanish at infinity.  In the special case of  locally compact Abelian groups, one can identify the Fourier and Fourier-Stieltjes algebras with the $L^1$-algebra and the measure algebra of the dual group respectively. In addition to the Fourier and Fourier-Stieltjes algebras, one can define the Rajchman algebra associated with a locally compact group $G$ to be $B_0(G)=B(G)\cap C_0(G)$, which can be identified with the algebra of Rajchman measures on the dual group in the Abelian case. It is easy to see that the Rajchman algebra is indeed a Banach subalgebra of the Fourier-Stieltjes algebra.

Let $\pi$ be a continuous unitary representation of $G$ on a Hilbert space ${\cal H}_\pi$. Let $A_\pi(G)$ denote the closed subspace of $B(G)$
generated by the coefficient functions of $G$ associated with $\pi$, i.e.
$$A_\pi=\overline{\Span_{\mathbb C}\{\xi*_\pi\eta:\xi,\eta\in {\cal H}_\pi\}}^{\|\cdot\|_{B(G)}}.$$
It is easy to see that $A_\pi(G)$ is a left and right translation-invariant closed subspace of $B(G)$. Conversely, by Theorem (3.17) of
\cite{arsac}, any closed subspace of $B(G)$ which is left and right translation-invariant, is of the form $A_\pi(G)$ for some
continuous unitary representation $\pi$.
Elements of $A_\pi(G)$ have certain forms as described below.
\begin{theorem}\label{Api-arsac}\cite{arsac}
The Banach space $A_\pi(G)$ consists of all the elements $u$ in $B(G)$ which are of the form
$$u=\sum_{i=1}^\infty \xi_n*_\pi\eta_n,$$
where $\xi_n$ and $\eta_n$ belong to ${\cal H}_\pi$ and $\sum_{i=1}^\infty\|\xi_i\|\|\eta_i\|<\infty$.
Moreover, for every $u$ in $A_\pi(G)$,
$$\|u\|_{B(G)}=\inf\{\sum_{i=1}^\infty\|\xi_i\|\|\eta_i\|: u \mbox{ represented as above}\},$$
and the infimum is attained.
\end{theorem}

Let ${\cal A}$ be a Banach algebra, and $X$ be a Banach ${\cal A}$-bimodule.
The dual space $X^*$ can be naturally equipped with dual module actions as below.
$$f\cdot a(x)=f(a\cdot x)\ \mbox{ and } \ a\cdot f(x)=f(x\cdot a), \ \mbox{ for }\ f\in X^*, \ a\in{\cal A}.$$
Then $X^*$ is an ${\cal A}$-bimodule, called a \emph{dual bimodule}.
A bounded linear map $D$ from ${\cal A}$ to an ${\cal A}$-bimodule $X$ is called a \emph{derivation} if for all $a$ and $b$ in ${\cal A}$,
$$D(ab)=D(a)\cdot b+a\cdot D(b).$$
A Banach algebra  ${\cal A}$ is called \emph{amenable} if
every continuous derivation $D$ from ${\cal A}$ to a dual ${\cal A}$-bimodule $X^*$ is inner, i.e. it is of the form
$$D:{\cal A}\rightarrow X^*,\  D(a)=a\cdot x-x\cdot a,$$
for an element $x$ of $X^*$.

A Banach algebra ${\cal A}$ is called a \emph{completely contractive
Banach algebra} if ${\cal A}$ has an operator space structure for which
the multiplication map $m:{\cal A}\times{\cal A}\rightarrow {\cal
A}$ is a completely contractive bilinear map.
Similarly, an operator space $X$ is a \emph{completely contractive ${\cal
A}$-bimodule} if
$X$ is an ${\cal A}$-bimodule, and the left and right module actions extend to
completely contractive maps on ${\cal A}\hat{\otimes} X$ and
$X\hat{\otimes} {\cal A}$ respectively.
Operator amenability is defined analogue to amenability. A completely contractive Banach algebra ${\cal A}$ is \emph {operator amenable} if every completely bounded derivation from ${\cal A}$ to a completely contractive dual ${\cal A}$-bimodule is inner.
One can refer to \cite{ruanbook} for more information on operator spaces.
\subsection{Induced representations}\label{section: induced-rep}
The most important method for producing representations is to induce representations for $G$ from representations of its subgroups.
The resulting representation is called an \emph{induced representation}.
Let $H$ be a closed subgroup of a locally compact group $G$, and $q$ be the quotient map from $G$ to $G/H$.
Assume that the quotient space $G/H$ admits a $G$-invariant measure
$\mu$, i.e. $\Delta_G|_H=\Delta_H$. One can always normalize the invariant measure $\mu$ on $G/H$ such that
for every $f$ in $C_c(G)$,
\begin{equation}\label{eq:haar-quotient}
\int_{G/H}\int_H f(xh)dhd\mu(xH)=\int_Gf(x)dx,
\end{equation}
where $dx$ and $dh$ denote the Haar measures of $G$ and $H$ respectively.
Let $\pi$ be a unitary representation
of $H$ on a Hilbert space ${\cal H}_\pi$.  To define the induced representation $\Ind_H^G\pi$ of $G$, we first define the new Hilbert space ${\cal F}$ to be
$\|\cdot\|_{{\cal F}_0}$-completion of the set ${\cal F}_0$, where
$${\cal F}_0=\left\{f\in C(G,{\cal H}_\pi):q(\supp f) \ \mbox{ is
compact}\ \& f(xh)=\pi(h^{-1})f(x) \ \forall x\in G, h\in H\right\},$$
and for every $f$ in ${\cal F}_0$,
$$\|f\|_{{\cal F}_0}^2=\int_{G/H}\|f(x)\|_{{\cal H}_\pi}^2d\mu(xH).$$
The induced representation  $\Ind_H^G\pi$ of $G$ is now defined to be left translations of functions in ${\cal F}$.

We usually work with a total subset of ${\cal F}$ described as follows. Let $C_c(G,{\cal H}_\pi)$ denote the set of continuous compactly supported ${\cal H}_\pi$-valued functions on $G$. Then the mapping
$${\cal P}:C_c(G,{\cal H}_\pi)\rightarrow C(G,{\cal H}_\pi), \qquad ({\cal P}f)(x)=\int_{H}\pi(h)f(xh)dh$$
is well-defined, and  ${\cal P}(C_c(G,{\cal H}_\pi))={\cal F}_0$.
Let $\alpha$ be a function in $C_c(G)$, and $\xi$ be a vector in ${\cal H}_\pi$. We define $f_{\alpha,\xi}$ to be
$$f_{\alpha,\xi}(x)=\alpha(x)\xi\quad \forall x\in G.$$
It is easy to see that $f_{\alpha,\xi}$ is a compactly supported ${\cal H}_\pi$-valued function, and the set
$$\{{\cal P}(f_{\alpha,\xi}): \alpha\in C_c(G), \xi\in {\cal H}\}$$
is a total subset of ${\cal F}$.

The reader may refer to \cite{folland} for details on basic representation theory of locally compact groups, and to \cite{volkerbook} for the theory of amenable
Banach algebras.
\section{Locally compact Abelian groups}\label{section:abelian}
Throughout this section, let $G$ be a locally compact Abelian group, and $M_0(G)$ be its algebra of Rajchman measures.
Let $M_c(G)$ denote the set of all continuous measures in $M(G)$, i.e. the set of complex bounded Radon measures
on $G$ which annihilate every singleton.
Clearly $M_c(G)$ and $M_0(G)$ are closed ideals of $M(G)$. Moreover, $M_0(G)$ is a subset of $M_c(G)$.
For a measurable subset $E$ of $G$, define $M_c(E)$ (respectively $M_0(E)$) to be the set of all measures in $M_c(G)$ (respectively $M_0(G)$) which are supported in the set $E$.

Let $G$ be a locally compact Abelian group,
and $P$ be a subset of $G$.
Let $k(P)$, called the torsion of $P$, denote the smallest positive integer $k$ such
that $\{kx:x\in P\}=\{0_G\}$, if such an integer exists. Otherwise, set
$k(P)=\infty$.
The set $P$ is called \emph{strongly independent} if for any positive integer $N$, any family $\{p_j\}_{j=1}^N$
of distinct elements of $P$, and any family of integers
$\{n_j\}_{j=1}^N$, the equality $\sum_{j=1}^N n_jp_j=0_G$ implies
that $n_j$ is a multiple of $k(P)$ for each $1\leq j\leq N$, unless $k(P)=\infty$, in which case $n_j=0$ for each $1\leq j\leq N$.
A \emph{reduced sum} on a strongly independent subset $P$ of torsion $k(P)$
is a formal expression $\sum_{i\in I}\dot{n_i}p_i$, where $I$ is a possibly empty finite index set,  $p_i$'s are distinct elements of $P$, and
$$0\neq \dot{n_i}\in \mathbb{Z}({\rm  mod }\  k(P)).$$
It is easy to see that every $x$ in ${\rm Gp}(P)$ can be expressed uniquely (up to equivalence) as a reduced sum.

Finally, recall that a closed subspace $B$ of $M(G)$ is called an \emph{$L$-space} if it satisfies the following condition.
\begin{center}
If $\mu,\nu\in  M(G)$, $\nu\in B$, and $\mu\ll\nu$, then $\mu\in B$.
\end{center}
Note that the above definition of an $L$-space is equivalent to the definition of a band, which has been used in \cite{varopoulos}.

In his rather difficult and technical paper \cite{varopoulos2}, Varopoulos proved that if $G$ is a non-discrete metrisable locally compact Abelian group, then there exists a perfect strongly independent subset $P$ of $G$
such that $M_0^+(P)\neq \{0\}$. The following lemma shows that one can assume that such a subset $P$ is compact as well.

\begin{lemma}\label{lem:compact-strongly-indep-M_0nonzero}
Let $G$ be a non-discrete metrisable locally compact Abelian group. Then there exists a compact perfect strongly independent subset $P$ of $G$
such that $M_0^+(P)\neq \{0\}$.
\end{lemma}
\begin{proof}
By \cite{varopoulos2}  there exists a perfect metrisable strongly independent subset  $P'$ of $G$ which supports a nonzero Rajchman measure $\mu_0$. It is known that $M_0(G)$ is an $L$-space \cite{GrahamM0}. Therefore, without loss of generality
we can assume that $\mu_0$ is a positive measure. Note that $\mu_0(P')>0$ and $\mu_0$ is a Radon
measure, therefore there exists a compact subset $K$ of $P$ with $\mu_0(K)>0$.
But $\mu_0|_K$ belongs to  $M_0(K)=M_0(G)\cap M(K)$, because it is a positive measure supported in $K$ and dominated by $\mu_0$.  Note that $\supp(\mu_0)$ is still a perfect set, because $\mu_0$ is a continuous  measure.
Let $P=\supp(\mu_0)$. Clearly $P$ is a strongly independent set, since it is a subset of the strongly independent set $P'$.
Hence $P$ is a compact perfect strongly independent
subset of $G$ with $M_0(P)\neq \{0\}$.
\end{proof}
Two measures $\mu$ and $\nu$ are said to be \emph{mutually singular}, denoted by $\mu\bot \nu$,
if  there exists a partition $A\cup B$ of $G$ such that $\mu$ is concentrated in $A$ and $\nu$ is concentrated in $B$.

\begin{lemma}\label{lem:Mc2&P}
Let $G$ be a non-discrete metrisable locally compact Abelian group, and
$P$ be a compact perfect strongly independent subset of $G$ such that $M_0(P)\neq \{0\}$. Then
\begin{itemize}
\item[1.] If $x$ and $y$ are distinct elements of $G$ then $M_0(x+P)\perp M_0(y+P)$.
\item[2.] For each $\mu$ in $M_c(G)$, we have $\sum_{x\in G}\mu(x+P)<\infty$.
\end{itemize}
\end{lemma}
\begin{proof}
\begin{itemize}
\item [1.] First note that if $x$ and $y$ are distinct elements of $G$ then $|(x+P)\cap(y+P)|\leq 2$. Indeed, assume that there exist distinct
elements $z_1$ and $z_2$ in $(x+P)\cap(y+P)$. Then there are $p_1,p_2,p'_1,$ and $p'_2$ in $P$ such that
$$z_1=x+p_1=y+p'_1\mbox{ and } z_2=x+p_2=y+p'_2,$$
which implies that $x-y=p'_1-p_1=p'_2-p_2$. Therefore $x-y$ should be an element of $P-P$.
Note that since $z_1\neq z_2$ and $x\neq y$, we have
$$p_1\neq p_2, \quad p'_1\neq p'_2,\quad p_1\neq p'_1,\quad p_2\neq p'_2.$$
Note that the element $x-y$ in $P-P$ can be expressed uniquely (up to permutation) as a reduced sum on $P$, i.e. one of the following cases happens:

{\bf Case 1:} $p'_1=p'_2$ and $p_1=p_2$, which is a contradiction with $z_1\neq z_2$.

{\bf Case 2:} $p'_1=-p_2$ and $p'_2=-p_1$, and $x-y=-p_1-p_2$ is the unique representation of $x-y$ in $-P-P$. Taking permutations into account,
there are at most two possibilities for $p_1$ and $p_2$, which implies that $$|(x+P)\cap(y+P)|\leq 2.$$
Therefore, every continuous measure $\mu$ on $G$ treats the sets $x+P$ as disjoint sets, i.e. $\mu((x+P)\cap(y+P))=0$ for distinct elements $x$ and $y$ in $G$. Thus $M_c(x+P)\perp M_c(y+P)$, and in particular $M_0(x+P)\perp M_0(y+P)$.

\item[2. ] For any finite number of points $x_1,\ldots,x_n$ in $G$,
$$\sum_{i=1}^n|\mu(x_i+P)|\leq|\mu|(\cup_{i=1}^n(x_i+P))\leq|\mu|(G)<\infty.$$
Hence,
$$\sum_{x\in G}|\mu(x+P)|=\sup_{I\subset G,|I|<\infty}\sum_{x\in I}|\mu(x+P)|\leq |\mu|(G)<\infty.$$
\end{itemize}
\end{proof}
\begin{theorem}\label{thm:M0(Abelian)-no-b.a.i}
Let $G$ be a non-discrete locally compact Abelian group. Then $M_0(G)$ cannot have an approximate identity.
\end{theorem}
\begin{proof}
First assume that $G$ is metrisable, and let $P$ be a compact perfect metrisable strongly independent subset of $G$ such that $M_0(P)\neq \{0\}$.
By Lemma \ref{lem:Mc2&P}, the sum $\sum_{x\in G}\mu(x+P)$ is convergent for any measure $\mu$ in $M_0(G)$.
Therefore only for at most countably many $x$ in $G$, $\mu(x+P)$ is nonzero. Let $\nu$ be an element of $M_0(G)$, and note that the function $x\mapsto \mu(x+P)$ is equal to 0 $\nu$-a.e. Therefore,
\begin{eqnarray*}
\mu*\nu(P)&=&\int_G\int_G\chi_P(x+y)d\mu (x)d\nu(y)=\int_G \mu(-y+P)d\nu(y)=0.
\end{eqnarray*}
Now suppose that $\{\mu_i\}_{i\in I}$ is an approximate identity of $M_0(G)$. Let $\mu_0$ be a nonzero measure in $M_0(P)$. Observe that
$\mu_i*\mu_0(P)=0$ and $\mu_0(P)>0$. This implies that the net $\{\mu_i*\mu_0\}_{i\in I}$ does not converge to $\mu_0$. Therefore $M_0(G)$ does not have an
approximate identity.

For a general non-discrete locally compact Abelian group $G$, let $H$ be a compact subgroup of $G$ such that $G/H$ is metrisable and non-discrete.
Let $q$ denote the quotient map from $G$ to $G/H$. The map $q$ induces an isometric Banach space isomorphism $\check{q}:C_0(G/H)\rightarrow C_0(G)$ via
composition by $q$. Define the map $\phi$ from $M(G)$ to $M(G/H)$ to be
$$\int_{G/H} fd\phi(\mu)=\int_G f\circ q d\mu,\quad \mbox{ for all } f\in C_0(G/H).$$
Clearly $\phi$, which is in fact the dual of $\check{q}$, is a surjective norm-decreasing Banach algebra homomorphism. Moreover, $\phi(M_0(G))=M_0(G/H)$.

Suppose that $\{\mu_i\}_{i\in I}$ is an approximate identity of $M_0(G)$. Then the net
$\{\phi(\mu_i)\}_{i\in I}$ is an approximate identity of $M_0(G/H)$. Indeed,
for any $\nu$ in $M_0(G/H)$ there exists $\mu$ in $M_0(G)$
such that $\phi(\mu)=\nu$, and
\begin{eqnarray*}
\lim_i \|\nu *\phi(\mu_i)-\nu\|_{M_0(G/H)}=\lim_i
\|\phi(\mu *\mu_i-\mu)\|_{M_0(G/H)}\leq \lim_i
\|(\mu*\mu_i-\mu)\|_{M_0(G)}=0.
\end{eqnarray*}
This contradicts the fact that the Rajchman algebra of a metrisable non-discrete locally compact Abelian group cannot have an approximate identity.
\end{proof}

Let $G$ be a non-compact locally compact Abelian group. Then the dual group $\hat{G}$ is non-discrete. Recall that the existence of a bounded approximate identity is a necessary condition for (operator)  amenability of a (completely contractive) Banach algebra.
Hence Theorem \ref{thm:M0(Abelian)-no-b.a.i} implies the following corollary.

\begin{corollary}\label{cor:B_0(abelian-noncomp)Not-opr-amen}
Let $G$ be a non-compact locally compact Abelian group. Then $B_0(G)$ is not (operator) amenable.
\end{corollary}

\section{Functorial properties of $B_0(G)$}\label{section:functorial}
Let $G$ be a locally compact group and $H$ be a closed subgroup of
$G$. Then the set of restrictions $B_0(G)|_H$ is a  subspace
of $B_0(H)$, which we will show is also closed. The extension problem asks whether every function in
$B_0(H)$ has an extension  in $B_0(G)$.
It has been proved that for every closed subgroup $H$ of a locally compact group $G$, one has $A(G)|_H=A(H)$ (see \cite{takesaki-masamichi} or \cite{herz}). In fact, every
function in the Fourier algebra of $H$ can be extended to a function of the same norm in the Fourier algebra of $G$. Unfortunately, the analogue of this result
does not hold  in general for the Fourier-Stieltjes algebra. However, for a locally compact group $G$ and a closed subgroup $H$, it has
been shown that $B(H)=B(G)|_H$ if $G$ is Abelian, or if $H$ is open,
or compact, or the connected component of the identity or the center
of $G$.
Moreover, Cowling and Rodway \cite{cowlingrodway} answered the extension problem of the Fourier-Stieltjes algebras in affirmative for the case of SIN-groups.
In this section, we prove similar results for Rajchman algebras.
The proof of Theorem \ref{thm1} is motivated by
that of Cowling and Rodway \cite{cowlingrodway}.

Let us first observe that the restriction map $r:B_0(G)\rightarrow B_0(H)$ is not surjective in the case when
$G=ax+b$ and  $H=\left\{\left(
                                \begin{array}{cc}
                                  1 & b \\
                                  0 & 1 \\
                                \end{array}
                              \right):
b\in {\mathbb R}\right\}\simeq {\mathbb R}$. Indeed,  Khalil \cite{Khalil} showed that $B_0(G)=A(G)$. Using the functorial properties of the Fourier algebra together with the fact that
$B_0({\mathbb R})\neq A({\mathbb R})$, it is clear that $B_0(G)|_H\neq B_0(H)$. In a consequent paper, we show that the restriction map between the Rajchman algebras of  $\SL$ and its certain subgroups fail to be surjective as well.

Proposition \ref{prop:B_0(G)|_H-inside-B_0(H)} follows form Proposition 2.10 of \cite{arsac}. In order to be self-contained, we include the proof in the present paper.
\begin{proposition}\label{prop:B_0(G)|_H-inside-B_0(H)}
Let $H$ be a closed subgroup of a locally compact group $G$. Then $B_0(G)|_H$ is a closed subspace of $B_0(H)$, and for each $u$ in $B_0(G)$,
$$\|u|_H\|_{B_0(H)}\leq\|u\|_{B_0(G)}.$$
\end{proposition}
\begin{proof}
Note that $B_0(G)$ is a translation-invariant closed subspace of $B(G)$. Therefore there
exists a unitary representation $\pi$ of $G$ such that $B_0(G)=A_\pi(G)$.
The space  $B_0(G)|_H$ is clearly a subspace of $B_0(H)$, since any representation of
$G$ restricts to a representation of $H$.  Next, observe  that $A_\pi(G)|_H=A_{\pi|_H}(H)$.
Indeed, let $u$ be an element of $B_0(G)$. Then by Theorem \ref{Api-arsac} (ii)
$$u=\sum_{i=1}^\infty \xi_n*_\pi\eta_n,$$
where $\xi_n$ and $\eta_n$ belong to ${\cal H}_\pi$ and $\sum_{i=1}^\infty\|\xi_i\|\|\eta_i\|<\infty$. Therefore
$$u|_H=\sum_{i=1}^\infty \xi_n*_{\pi|_H}\eta_n,$$
where $\pi|_H$ is the restriction of the representation $\pi$ from $G$ to $H$.
This implies that $u|_H$ belongs to $A_{\pi|_H}(H)$. On the other hand, let $v$ be an element of $A_{\pi|_H}(H)$. Applying Theorem \ref{Api-arsac}
(ii) again, we get
$$v=\sum_{i=1}^\infty \xi'_n*_{\pi|_H}\eta'_n,$$
where $\xi'_n$ and $\eta'_n$ belong to ${\cal H}_\pi$ and $\sum_{i=1}^\infty\|\xi'_i\|\|\eta'_i\|<\infty$. Define
$$w=\sum_{i=1}^\infty \xi'_n*_{\pi}\eta'_n.$$
Then $w$ belongs to $A_\pi(G)$ and $w|_H=v$. Hence $A_\pi(G)|_H=A_{\pi|_H}(H)$, and the latter is a closed subspace of $B(G)$ by definition.

Finally, for every $u$ in $B(G)$, we can find a representation
$u(x)=\langle \pi (x)\xi,\eta \rangle$ such that
$\|u\|_{B(G)}=\|\xi\|\|\eta\|$. Then $u|_H(h)=\langle\pi|_H
(h)\xi,\eta\rangle$, and $\|u|_H\|_{B(H)}\leq
\|\xi\|\|\eta\|=\|u\|_{B(G)}$.
\end{proof}

\begin{lemma}\label{lem:C0-indues-to-C0-rep}
Let $H$ be a closed subgroup of a locally compact group $G$ such that $\Delta_G|_H=\Delta_H$, and
$\pi$ be a representation of $H$. If  $A_\pi(H)$ is a subset of  $C_0(H)$ then $A_{Ind_\pi}(G)$ is a subset of $C_0(G)$ as well.
\end{lemma}
\begin{proof}
Let $\mu$ be the nonzero positive invariant measure of $G/H$ normalized such that
for every $f$ in $C_c(G)$,
\begin{equation}\label{eq:haar-quotient}
\int_{G/H}\int_H f(xh)dhd\mu(xH)=\int_Gf(x)dx.
\end{equation}
For any $\xi$ in ${\cal H}_\pi$ and $v$ in $C_c(G)$, we define the compactly supported function
$f_{v,\xi}:G\rightarrow {\cal H}_\pi$, $x\mapsto v(x)\xi$.
Let $\eta$ and $w$ be elements of ${\cal H}_\pi$ and $C_c(G)$ respectively,
and compute the coefficient function of $\Ind_H^G\pi$ corresponding to ${\cal P}f_{v,\xi}$ and ${\cal P}f_{w,\eta}$.
\begin{eqnarray}
\langle\Ind_\pi(x){\cal P}f_{v,\xi},{\cal P}f_{w,\eta}\rangle_{{\cal
F}_0}&=&\int_{G/H}\langle\Ind_\pi(x){\cal P}f_{v,\xi}(g),{\cal P}f_{w,\eta}(g)\rangle_{{\cal
H}_\pi}d\mu (gH)\nonumber\\&=& \int_{G/H}\langle \int_H
\pi(h)(v(x^{-1}gh)\xi)dh,\int_H \pi(h')(w(gh')\eta)dh'\rangle_{{\cal
H}_\pi}d\mu (gH)\nonumber\\&=& \int_{G/H} \int_H \int_H
v(x^{-1}gh)w(gh')\langle\pi(h)\xi,\pi(h')\eta\rangle_{{\cal
H}_\pi}dhdh'd\mu (gH)\nonumber\\&=&\int_{G/H} \int_H \int_H
v(x^{-1}gh)w(gh')\langle\pi(h'^{-1}h)\xi,\eta\rangle_{{\cal
H}_\pi}dhdh'd\mu (gH)\nonumber\\&=& \int_{G/H} \int_H \int_H
v(x^{-1}gh'h)w(gh')\langle\pi(h)\xi,\eta\rangle_{{\cal
H}_\pi}dhdh'd\mu (gH)\nonumber\\&=&\int_G\int_H
v(x^{-1}gh)w(g)\xi*_\pi\eta(h)dhdg,\label{eq:induced-rep-coeff}
\end{eqnarray}
where in the last equality, we used the normalized relation
stated in Equation (\ref{eq:haar-quotient}).
Observe that for each $x$ in $G$, the following integral is bounded.
\begin{eqnarray*}
\int_G\int_H|v(x^{-1}gh)w(g)|dhdg&=&\int_G P|v|(x^{-1}gH)|w|(g)dg
\leq \|P|v|\|_{L^\infty(G/H)} \|w\|_{L^1(G)}<\infty,
\end{eqnarray*}
since $P|v|$ belongs to $C_c(G/H)$.

Let
$\epsilon>0$ be given, and define
$\epsilon_1=\frac{\epsilon}{\|P|v|\|_{L^\infty(G/H)} \|w\|_{L^1(G)}}$. Since $\xi*_\pi\eta$ belongs to $C_0(H)$,
there exists a compact subset $K_1$ of $H$ such that
$|\xi*_\pi\eta(h)|<\epsilon_1$ for every $h$ in $H\setminus K_1$. Let $K=[\supp(w)]K_1[\supp(v)]^{-1}$, and note
that $K$ is compact. It is easy to see that if $x$ and $h$ are elements of  $G\setminus K$ and $K_1$ respectively, then  $v(x^{-1}gh)w(g)=0$ for
every $g$ in $G$. Hence, for each $x$ in $G\setminus K$,
\begin{eqnarray*}
|{\cal P}f_{v,\xi}*_{\Ind_\pi}{\cal P}f_{w,\eta}(x)|&=&|\int_G\int_H v(x^{-1}gh)w(g)\xi*_\pi\eta(h)dhdg|\\
&=&|\int_G\int_{H\setminus K_1} v(x^{-1}gh)w(g)\xi*_\pi\eta(h)dhdg|\\
&\leq&
\int_G\int_{H\setminus K_1} |v(x^{-1}gh)w(g)\xi*_\pi\eta(h)|dhdg\\
&\leq& \epsilon_1 \int_G\int_{H} |v(x^{-1}gh)w(g)|dhdg\leq \epsilon,
\end{eqnarray*}
which implies that ${\cal P}f_{v,\xi}*_{\Ind_\pi}{\cal P}f_{w,\eta}$ belongs to $C_0(G)$. Finally note that the set
$\{{\cal P}f_{v,\xi}: \ v\in C_c(G),\xi\in {\cal H}_\pi\}$ forms a total subset of the Hilbert space of the representation $\Ind_\pi$, therefore
$$A_{\Ind_\pi}(G)=\overline{\Span}^{\|\cdot\|_{B(G)}}\left\{{\cal P}f_{v,\xi}*_{\Ind_\pi}{\cal P}f_{w,\eta}: \ v,w\in C_c(G), \xi,\eta\in {\cal H}_\pi\right\}.$$
Hence, $A_{\Ind_\pi}(G)$ is a subset of $C_0(G)$ as well.
\end{proof}
A locally compact group is called a SIN-group if it has a neighborhood basis of the identity consisting of pre-compact neighborhoods which are invariant under inner automorphisms, i.e. their characteristic functions are central.  Recall that a function $\nu :G\rightarrow \mathbb{C}$ is called \emph{central} if for
every $g$ and $g'$ in $G$, we have
$\nu(gg')=\nu(g'g)$. Examples of SIN-groups are Abelian groups, compact groups, and discrete groups. It is easy to see that SIN-groups are unimodular.
Moreover, every closed subgroup of a SIN-group is again a SIN-group.
\begin{theorem}\label{thm1}
Let  $H$ be a closed subgroup of a SIN-group $G$. Then the restriction map $r:B_0(G)\rightarrow B_0(H)$ is surjective.
\end{theorem}
\begin{proof}
Let $dg$ and $dh$ denote the Haar measures of $G$ and $H$ respectively. Note
that $G/H$ admits a $G$-invariant measure $d\dot{g}$, since $\Delta_G|_H=\Delta_H$.
Moreover assume that these measures are
normalized so that Equation (\ref{eq:haar-quotient}) holds.
Let $\pi_0$ be a representation of $H$
satisfying $B_0(H)=A_{\pi_0}(H)$, and let $F_{\pi_0}(H)$ be the set of all the linear combinations of the coefficient functions of $\pi_0$.
By Lemma \ref{lem:C0-indues-to-C0-rep}, every coefficient function of the induced representation $\Ind_{\pi_0}$
of $G$ vanishes at infinity. Hence by Proposition \ref{prop:B_0(G)|_H-inside-B_0(H)} to prove that $r$ is surjective, it is enough to show that $A_{\Ind_{\pi_0}}(G)|_H$ is dense in $F_{\pi_0}(H)$, since $F_{\pi_0}(H)$ is in turn dense in $B_0(H)$.

Let  $\{{\cal U}_\alpha\}_\alpha$  be a basis of relatively compact neighborhoods of the identity in $G$.
For each $\alpha$, let ${\cal V}_\alpha$ be a relatively compact neighborhood of $e_G$ such that ${\cal V}_\alpha^{-1}{\cal V}_\alpha\subseteq {\cal U}_\alpha$.
Since $G$ is a SIN-group, there exists a net $\{v_\alpha\}_\alpha$ of nonnegative continuous central functions which are not identically zero and satisfy $\supp(v_\alpha)\subseteq {\cal V}_\alpha$. Observe that for each $\alpha$, the function $\check{v_{\alpha}}*v_{\alpha}$ is a nonnegative continuous
function supported in ${\cal U}_\alpha$. Moreover,
$$\check{v_\alpha}* v_\alpha(e)=\int_Gv_\alpha(y^{-1})v_\alpha(y^{-1}e)dy=\int_Gv_\alpha(y)v_\alpha(y)dy=\|v_\alpha\|_2^2>0.$$
Let $u_\alpha=\frac{\check{v_\alpha}*v_\alpha|_H}{\|\check{v_\alpha}*v_\alpha|_H\|_{L^1(H)}}$ for each $\alpha$. Then $\{{\cal U}_\alpha\cap H\}_\alpha$  is a basis of relatively compact neighborhoods of the identity in $H$, and the net $\{u_\alpha\}_\alpha$ is a bounded approximate identity of $L^1(H)$ consisting of nonnegative continuous functions (see Theorem
A.1.8. of \cite{volkerbook}). It is clear that for every function $g$ in $C(H)$, and for every vector $\vartheta$ in ${\cal H}_\pi$,  the following equations hold.
\begin{eqnarray}\label{eq:b.a.i.converge}
\lim_\alpha \int_H u_\alpha(h)g(h)dh= g(e)\nonumber\\
\lim_\alpha\|\pi_0(u_\alpha)^*\vartheta-\vartheta\|= 0.
\end{eqnarray}

Fix an element $f=\xi*_{\pi_0}\eta$ of $F_{\pi_0}(H)$, and consider the function $F_{\alpha}$ defined as in Equation \ref{eq:induced-rep-coeff}.
\begin{equation}\label{eq-Ind-coeff}
F_{\alpha}(x)=\langle\Ind_{\pi_0}(x){\cal P}f_{v'_\alpha,\xi},{\cal P}f_{v'_\alpha,\eta}\rangle_{{\cal
F}_0}=\int_G\int_H v'_\alpha(x^{-1}gh)v'_\alpha(g)\xi*_{\pi_0}\eta(h)dhdg,\nonumber
\end{equation}
where $v'_\alpha=\frac{1}{\sqrt{\|\check{v_\alpha}*v_\alpha|_H\|_{L^1(H)}}}v_\alpha$.
For every $h'$ in $H$, we have
\begin{eqnarray*}
F_{\alpha}|_H(h')&=&\int_G\int_H v'_{\alpha}(h'^{-1}gh)v'_\alpha(g)\xi*_{\pi_0}\eta(h)dhdg\\
&=&\int_G\int_H v'_\alpha(ghh'^{-1})v'_\alpha(g)\xi*_{\pi_0}\eta(h)dhdg\\
&=&\int_G\int_H v'_\alpha(gh)v'_\alpha(g)\xi*_{\pi_0}\eta(hh')dhdg\\
&=&\int_G\int_Hv'_\alpha(g^{-1}h)v'_\alpha(g^{-1})\langle\pi_0(h')\xi,\pi_0(h^{-1})\eta\rangle dhdg\\
&=&\langle\pi_0(h')\xi,\pi_0( \check{v'_\alpha}*v'_{\alpha}|_H)^*\eta\rangle\\
&=&\langle\pi_0(h')\xi,\pi_0(u_\alpha)^*\eta\rangle,\\
\end{eqnarray*}
where we used the unimodularity of $G$ and $H$. Note that
$$\|f-F_\alpha|_H\|_{B(H)}=\|\xi*_{\pi_0}\eta-\xi*_{\pi_0}\pi_0(u_\alpha)^*\eta\|_{B(H)}
\leq\|\xi\|\|\eta-\pi_0(u_\alpha)^*\eta\|\rightarrow 0,$$
by (\ref{eq:b.a.i.converge}). Hence $A_{\Ind_{\pi_0}}(G)|_H$ is dense in $F_{\pi_0}(H)$ and we are done.
\end{proof}

\begin{theorem}\label{thm2}
Let $G$ be a locally compact group. Then
the restriction map $r:B_0(G)\rightarrow B_0(H)$ is surjective in the following cases.
\begin{enumerate}
\item  When $H$ is the center of $G$.
\item When $H$ is an open subgroup of $G$
\end{enumerate}
\end{theorem}
\begin{proof}
Let $dg$ and $dh$ denote the Haar measures of $G$ and $H$ respectively. In either of the above cases,
the quotient space $G/H$ admits a $G$-invariant measure $d\dot{g}$ that can be normalized to satisfy Equation (\ref{eq:haar-quotient}).
Define $\pi_0$, $A_{\pi_0}(H)$, and $F_{\pi_0}(H)$ as in the proof of Theorem \ref{thm1}.
Again, we only need to show that $A_{\Ind_{\pi_0}}(G)|_H$ is dense in $F_{\pi_0}(H)$.
\\

{\bf Case 1:} Suppose that $H$ is the center of $G$. Let  $\{{\cal U}_\alpha\}_\alpha$ and $\{{\cal V}_\alpha\}_\alpha$ be bases of relatively compact neighborhoods of $e_G$ in $G$ such that ${\cal V}_\alpha{\cal V}_\alpha^{-1}\subseteq {\cal U}_\alpha$ for each $\alpha$.
Fix a net $\{v_\alpha\}_\alpha$ of nonnegative continuous functions which are not identically zero and satisfy $\supp(v_\alpha)\subseteq {\cal V}_\alpha$. Observe that for each $\alpha$, the function $u_\alpha=\frac{v_\alpha*\check{v_\alpha}|_H}{\|v_\alpha*\check{v_\alpha}|_H\|_{L^1(H)}}$ is a nonnegative continuous
function supported in ${\cal U}_\alpha\cap H$, and the net
$\{u_\alpha\}_\alpha$ forms a bounded approximate identity for $L^1(H)$.
Fix an element $f=\xi*_{\pi_0}\eta$ of $F_{\pi_0}(H)$, and consider the function $F_{\alpha}$ defined as
\begin{equation}\label{eq-Ind-coeff}
F_{\alpha}(x)=\langle\Ind_{\pi_0}(x){\cal P}f_{v'_\alpha,\xi},{\cal P}f_{v'_\alpha,\eta}\rangle_{{\cal
F}_0}=\int_G\int_H v'_\alpha(x^{-1}gh)v'_\alpha(g)\xi*_{\pi_0}\eta(h)dhdg,\nonumber
\end{equation}
where $v'_\alpha=\frac{1}{\sqrt{\|v_\alpha*\check{v_\alpha}|_H\|_{L^1(H)}}}v_\alpha$.

For every $h'$ in $H$, we have
\begin{eqnarray*}
F_{\alpha}|_H(h')&=&\int_G\int_H v'_{\alpha}(h'^{-1}gh)v'_\alpha(g)\xi*_{\pi_0}\eta(h)dhdg\\
&=&\int_G\int_H v'_\alpha(gh'^{-1}h)v'_\alpha(g)\xi*_{\pi_0}\eta(h)dhdg\\
&=&\int_G\int_H v'_\alpha(gh)v'_\alpha(g)\xi*_{\pi_0}\eta(h'h)dhdg\\
&=&\int_G\int_Hv'_\alpha(g)\check{v'_\alpha}(g^{-1}h^{-1})\langle\pi_0(h')\xi,\pi_0(h^{-1})\eta\rangle dhdg\\
&=&\int_H[v'_\alpha*\check{v'_\alpha}](h^{-1})\langle\pi_0(h')\xi,\pi_0(h^{-1})\eta\rangle dh\\
&=&\langle\pi_0(h')\xi,\pi_0( v'_{\alpha}*\check{v'_\alpha}|_H)\eta\rangle,\\
\end{eqnarray*}
where we used the unimodularity of $H$. Note that
$$\|f-F_\alpha|_H\|_{B(H)}=\|\xi*_{\pi_0}\eta-\xi*_{\pi_0}\pi_0(u_\alpha)\eta\|_{B(H)}
\leq\|\xi\|\|\eta-\pi_0(u_\alpha)^*\eta\|\rightarrow 0,$$
by (\ref{eq:b.a.i.converge}). Hence $A_{\Ind_{\pi_0}}(G)|_H$ is dense in $F_{\pi_0}(H)$.\\

{\bf Case 2:} Assume that $H$ is an open subgroup of $G$, and $dh$ is the restriction of $dx$ to the open subset $H$. Let  $\{{\cal U}_\alpha\}_\alpha$  be a basis of relatively compact neighborhoods of $e_G$ in $G$ such that ${\cal U}_\alpha\subseteq H$ for each $\alpha$. Fix a net $\{u_\alpha\}_\alpha$ of nonnegative continuous functions which are not identically zero and satisfy $\supp(u_\alpha)\subseteq {\cal U}_\alpha$. Observe that the net
$\{u_\alpha\}_\alpha$ forms a bounded approximate identity for both $L^1(H)$ and $L^1(G)$.
Fix an element $f=\xi*_{\pi_0}\eta$ of $F_{\pi_0}(H)$, and consider the function $F_{\alpha}$ defined as
\begin{equation}\label{eq-Ind-coeff}
F_{\alpha}(x)=\langle\Ind_{\pi_0}(x){\cal P}f_{u_\alpha,\xi},{\cal P}f_{u_\alpha,\eta}\rangle_{{\cal
F}_0}=\int_G\int_H u_\alpha(x^{-1}gh)u_\alpha(g)\xi*_{\pi_0}\eta(h)dhdg.\nonumber
\end{equation}
For every $h'$ in $H$, we have
\begin{eqnarray*}
F_{\alpha}|_H(h')&=&\int_G\int_H u_{\alpha}(h'^{-1}gh)u_\alpha(g)\xi*_{\pi_0}\eta(h)dhdg\\
&=&\int_H\int_H u_{\alpha}(h'^{-1}h''h)u_\alpha(h'')\xi*_{\pi_0}\eta(h)dhdh''\\
&=&\int_H\int_H u_{\alpha}(h'^{-1}h)u_\alpha(h'')\xi*_{\pi_0}\eta(h''^{-1}h)dhdh''\\
&=&u_\alpha*(\xi*_\pi\eta)*\check{u_\alpha}(h')\\
&=&[\pi_0(u_\alpha)\xi]*_{\pi_0}[\pi_0(u_\alpha)\eta].\\
\end{eqnarray*}
Note that
$$\|f-F_\alpha|_H\|_{B(H)}=\|\xi*_{\pi_0}\eta-[\pi_0(u_\alpha)\xi]*_{\pi_0}[\pi_0(u_\alpha)\eta]\|_{B(H)}
\rightarrow 0,$$
by Equation (\ref{eq:b.a.i.converge}). Hence $A_{\Ind_{\pi_0}}(G)|_H$ is dense in $F_{\pi_0}(H)$ and we are done.

\end{proof}

One of the most natural questions about $B_0(G)$ is to characterize the groups $G$ for which the Rajchman algebra properly contains the Fourier algebra.
In 1966,
Hewitt and Zuckerman \cite{Hewitt-Zuckerman} proved that the inclusion of $A(G)$ in $B_0(G)$ is proper
for every non-compact locally compact Abelian group $G$. On the other hand, in his study of the representations of $ax+b$ group, Khalil \cite{Khalil} proved that the Rajchman algebra and the Fourier algebra coincide in this case. The question is open in general.

A locally compact group $G$ is called an AR-group if the left regular representation of $G$ decomposes into a direct sum of irreducible representations. Clearly ${\mathbb R}$ is not an AR-group. On the other hand, compact groups and $ax+b$ group are examples of AR-groups.
Fig{\`a}-Talamanca proved that if $G$ is a unimodular non-compact locally compact group for which $A(G)=B_0(G)$, then $G$ is an AR-group (\cite{Figa1} and \cite{Figa1}). In \cite{Baggett-Taylor-A=B0=>reducible}, Baggett and Taylor showed that the above result holds even without the unimodularity condition. This result together with Theorem 3.1 of \cite{miao-mah} implies that $B_0(G)$ is larger than $A(G)$ for any non-compact IN-group $G$.
As a corollary of Theorem \ref{thm1}, we can easily obtain this result for the special case of non-compact connected SIN-groups.

\begin{corollary}
Let $G$ be a non-compact connected SIN group. Then $A(G)\neq B_0(G)$
\end{corollary}
\begin{proof}
Assume not, i.e. $A(G)=B_0(G)$. Since $G$ is a non-compact connected  SIN group, it has a non-compact Abelian closed subgroup $H$.
By the surjective-ness of the restriction map for SIN groups we have,
$$A(H)=A(G)|_H=B_0(G)|_H=B_0(H),$$
where we have used the fact that for any locally compact group $G$ and its closed subgroup $H$, the restriction map between the Fourier algebras are onto. But $A(H)\neq B_0(H)$ for any non-compact Abelian group.
\end{proof}

\section{Amenability properties of Rajchman algebras}\label{section:amenability-of-B0}
In this section, we use the following hereditary property of (operator) amenability to extend the results of Section \ref{section:abelian}.
Let ${\cal A}$ and ${\cal B}$ be (completely contractive) Banach algebras, and $\phi$ be a surjective (completely) bounded homomorphism from ${\cal A}$ to ${\cal B}$ . If
${\cal A}$ is (operator) amenable then ${\cal B}$ is (operator) amenable as well (see Section 2.3 of \cite{volkerbook}). It is very easy to see that a similar property holds for
the existence of a bounded approximate identity, i.e.

\begin{proposition}\label{prop:hereditary-bai}
Let ${\cal A}$ and ${\cal B}$ be Banach algebras, and $\phi$ be a surjective bounded homomorphism from ${\cal A}$ to ${\cal B}$ . If
$\{u_\alpha\}_{\alpha\in I}$ is a bounded approximate identity for ${\cal A}$ then
$\{\phi(u_\alpha)\}_{\alpha\in I}$ is a bounded approximate identity for ${\cal B}$.
\end{proposition}
Moreover, a closed ideal of an amenable Banach algebra is amenable if it is weakly complemented.
Finally recall that the existence of a bounded approximate identity is a necessary condition for (operator)  amenability of a (completely contractive) Banach algebra. The reader may refer to \cite{volkerbook} for more details.
We are now able to characterize the groups $G$ whose Rajchman algebras are amenable.
\begin{theorem}\label{thm:B0-amen}
Let ${\cal A}$ be a closed subalgebra of $B(G)$ which contains
$B_0(G)$. Then ${\cal A}$ is amenable if and only if $G$ is compact and has an
Abelian subgroup of finite index.
\end{theorem}
\begin{proof}
Suppose $G$ is compact and has an Abelian subgroup of finite index.
Then $B_0(G)={\cal A}=B(G)$, and it is amenable by Corollary 4.2 of \cite{lau-roy-willis}.

Conversely, suppose that ${\cal A}$ is amenable. Since $B_0(G)$ and
$A(G)$ are complemented ideals of ${\cal A}$, they are amenable as
well. Hence, by the characterization of amenable Fourier algebras by Forrest and Runde \cite{forrest-runde-A}, $G$
is almost Abelian, i.e. it has an Abelian subgroup $H$ of finite
index. Note that $H$ is clearly an open subgroup. Hence by Theorem \ref{thm2} the
restriction map $r:B_0(G)\rightarrow B_0(H)$ is surjective, which implies
that $B_0(H)$ is amenable as well. Since $H$ is Abelian, by Corollary \ref{cor:B_0(abelian-noncomp)Not-opr-amen} the amenability
of $B_0(H)$ implies that $H$ is compact.  Therefore
$G$ is compact as well.
\end{proof}

 The question of characterizing the operator amenability of $B_0(G)$ turns out to be difficult. In the rest of this paper, we prove extreme cases for operator amenability of $B_0(G)$.
First note that if $G$ is compact then $B_0(G)=B(G)=A(G)$, and $B_0(G)$ is  operator amenable \cite{ruan}.
For certain groups such as Fell groups and the $ax+b$ group, the associated  Rajchman algebras are non-amenable, but they are operator amenable. Indeed,  it has been shown for both cases that the Rajchman algebra and the Fourier algebra coincide (see \cite{VolkerNicoOpr2} and \cite{Khalil}), and are operator amenable due to the amenability of the groups themselves \cite{ruan}.
On the other extreme, we show that the Rajchman algebra of a connected non-compact SIN-group and  $\SL$ cannot
be operator amenable.

\begin{proposition}\label{prop:noncpt-connected-SIN-NOT-opr-amen}
Let $G$ be a non-compact connected SIN-group. Then,
 $B_0(G)$ does not have a bounded approximate identity, and is not operator amenable.
\end{proposition}
\begin{proof}
Since $G$ is a non-compact connected SIN-group, it is of
the form $G=\mathbb{R}^n\times K$, where $K$ is a compact subgroup.
Hence $\mathbb{R}^n$ is a closed subgroup of the SIN-group $G$, and
by Theorem \ref{thm2},  the restriction map $r:B_0(G)\rightarrow
B_0(\mathbb{R}^n)$ is a surjective completely bounded algebra homomorphism. Now suppose that $B_0(G)$ is
has a bounded approximate identity. Then by Proposition \ref{prop:hereditary-bai}
$B_0(\mathbb{R}^n)$  also has a bounded approximate identity, which contradicts Theorem \ref{thm:M0(Abelian)-no-b.a.i}. Hence $B_0(G)$ does not
have a bounded approximate identity, which implies that it is not operator amenable.
\end{proof}

\begin{remark}
Let $\GL$ and $\Hn$ denote the general linear group and the Heisenberg group of degree $n$ respectively.
Using an argument similar to the proof of Proposition \ref{prop:noncpt-connected-SIN-NOT-opr-amen}, one can see that the Rajchman algebras $B_0(\GL)$ and  $B_0(\Hn)$
do not have bounded approximate identities, and are not operator amenable. Indeed, we only need to consider the restriction map from each group to
its center, and notice that the center of each group is non-compact.
\end{remark}
\subsection{Special linear group}
Let us now consider the case $G=\SL$.
We use the notations from \cite{folland} and parametrize the dual space $\widehat{\SL}$ through its
identification with the following family of representations:
\begin{eqnarray*}
&&\text{trivial representation:} \quad \iota,\\
&&\text{principal continuous series:}\quad \{\pi_{it}^+:\ t\geq 0\}\cup\{\pi_{it}^-:\ t>0\},\\
&&\text{discrete series:}\quad\{\delta_{\pm n}:\ n\geq 2\},\\
&&\text{mock discrete series:}\quad \delta_{\pm 1},\\
&&\text{complementary series:} \quad \{\kappa_s:\ 0<s<1\}.\\
\end{eqnarray*}
We use the results of Repka \cite{repka} and Puk\'{a}nszky \cite{Pukanszky} regarding the decomposition of tensor products of unitary representations of $\SL$
to observe that $B_0(\SL)$ is not (operator) amenable. One may refer to \cite{folland}, \cite{dixmier} and \cite{arsac} for more details regarding direct integrals. The author would like to thank Viktor Losert for pointing her attention to the above-mentioned results.

\begin{theorem}\label{b0-SL-not-sq-dense}
If $G$ is the group $\SL$ then $B_0(G)$ is not square-dense, i.e.
$$\overline{B_0(G)^2}\neq B_0(G).$$
\end{theorem}
\begin{proof}
Let $\mu$ denote the Plancherel measure on $\widehat{\SL}$. Recall that the Plancherel measure of the complementary series,  mock discrete series,
and the trivial representation is zero. Moreover,
by Harish-Chandra's trace formula the Plancherel measure on the principal and discrete series is defined as
\begin{eqnarray*}
&& d\mu(\pi_{it}^+)=\frac{t}{2}\tanh \frac{\pi t}{2}\ dt,\\
&& d\mu(\pi_{it}^-)=\frac{t}{2}\coth \frac{\pi t}{2}\ dt,\\
&& \mu(\{\delta_{\pm n}\})=n-1.
\end{eqnarray*}
Therefore, by Proposition 8.4.4 of \cite{dixmier}, the left regular representation $\lambda$ is quasi-equivalent with the representation
\begin{equation}\label{eq:plancherel}
\int^\oplus_{(o,\infty)}\pi_{it}^+ dt\oplus\int^\oplus_{(o,\infty)}\pi_{it}^- dt\oplus\bigoplus_{n=2}^\infty(\delta_n\oplus\delta_{-n}).
\end{equation}
Let $m_{\hat{G}}$ denote the renormalized Plancherel measure given in (\ref{eq:plancherel}). Define the new representations
$$\Pi_0^+=\int^\oplus_{(o,\infty)}\pi_{it}^+ dt\oplus\bigoplus_{k=1}^\infty(\delta_{2k}\oplus\delta_{-2k}),$$
and
$$\Pi_0^-=\int^\oplus_{(o,\infty)}\pi_{it}^- dt\oplus\bigoplus_{k=1}^\infty(\delta_{2k+1}\oplus\delta_{-2k-1}),$$
and observe that the matrix coefficients $A_{\Pi_0^+}$ and $A_{\Pi_0^-}$ are contained in $A(G)$.
Note that these representations are used in the direct integral decomposition of tensor products of irreducible unitary representations of $\SL$.
In fact, Repka \cite{repka} proved that if $\pi$ and $\pi'$ are irreducible unitary representations of $\SL$ then
$$\pi\otimes\pi'\simeq_q\left\{\begin{array}{cc}
                                 \Pi_0^+\oplus\kappa_{r+s-1} & \text{if } \{\pi,\pi'\}=\{\kappa_r,\kappa_s\} \text{ and } r+s\geq 1 \\
                                 \Pi & \text{otherwise},
                               \end{array}
\right.$$
where $\Pi$ is a subrepresentation of $\Pi^+_0$ or $\Pi_0^-$, and $\simeq_q$ denotes the quasi-equivalence of representations.

 For irreducible unitary representations $\pi$ and $\pi'$ of $G$, let $m_{\pi,\pi'}$ denote the measure on $\hat{G}$ which appears in the direct integral decomposition of $\pi\otimes\pi'$. By \cite{repka}, $m_{\pi,\pi'}$ is absolutely continuous with respect to the Plancherel measure $m_{\hat{G}}$
 on $\hat{G}_r$, and $\supp(m_{\pi,\pi'})$ contains at most one element from the complementary series.
Now let $u$ and $u'$ be elements of the coefficient spaces $A_\pi$ and $A_{\pi'}$ respectively, with
trace operators $T_\pi$ and $T_{\pi'}$ such that
$$u=\Tr(\pi(\cdot)T_\pi) \quad \mbox{ and }\quad u'=\Tr(\pi'(\cdot)T_{\pi'}).$$
Then
\begin{equation}\label{eq:tensor-coeff-irred}
uu'=\Tr(\pi\otimes\pi'(\cdot)T_\pi\otimes T_{\pi'})=\int_{\hat{G}}\Tr(\pi''(\cdot)T_{\pi,\pi';\pi''})dm_{\pi,\pi'}(\pi'').
\end{equation}

Finally let $u$ and $u'$ be elements of $B_0(G)$. By Corollary 3.55 of \cite{arsac}, there exist  positive measures $\mu$ and $\mu'$ on $\hat{G}$
such that
\begin{eqnarray*}
u=\int_{\hat{G}}\Tr(\pi(\cdot)T_\pi)d\mu(\pi) \qquad \mbox{and} \qquad u'=\int_{\hat{G}}\Tr(\pi(\cdot)T'_\pi)d\mu'(\pi),
\end{eqnarray*}
where $\{T_\pi\}_{\pi\in\hat{G}}$ and
$\{T'_\pi\}_{\pi\in\hat{G}}$ are  elements of $L^1(\hat{G},\mu)^\oplus$ and $L^1(\hat{G},\mu')^\oplus$ respectively.
Therefore by (\ref{eq:tensor-coeff-irred}) we have,
\begin{eqnarray}\label{eq: uu'-integral-formula}
uu(\cdot)&=&\int_{\hat{G}\times\hat{G}}\Tr(\pi\otimes\pi'(\cdot)T_\pi\otimes T'_{\pi'})d\mu(\pi,\pi')\nonumber\\
&=&\int_{\hat{G}\times\hat{G}}\int_{\hat{G}} \Tr(\pi''(\cdot)T_{\pi,\pi';\pi''})dm_{\pi,\pi'}(\pi'') d\mu(\pi,\pi').
\end{eqnarray}

For a unitary representation $\pi$ of $G$, let $\tilde{\pi}$ denote the surjective map generated by $\pi$ from $\VN_\omega(G)$ to
$\VN_\pi(G)$, where $\omega$ is the universal representation of $G$. Note that every unitary representation $\pi$
of $G$ extends to a nondegenerate norm-decreasing $*$-representation of $C^*$-algebras from $C^*(G)$ to
$C_\pi^*(G)$, which identifies $C^*_\pi(G)$ with a quotient of $C^*(G)$. Then the dual map $\pi^*$ identifies $B_\pi(G)$ with  a subset
of $B(G)$, and we have
$$\tilde{\pi}=(\pi^*|_{A_\pi})^*.$$
Hence for every $S$ in $\VN_\omega(G)$, we have
$$\tilde{\pi}(S)=S|_{A_\pi}.$$
Now fix a positive real number  $t$. Then $\pi_{it}^+$ and $\oplus_{\pi\in \hat{G}\setminus\{\pi_{it}^+\}}\pi$ are disjoint unitary representations of $\SL$, and
by Proposition 3.12 of \cite{arsac}, $A_{\pi_{it}^+}$ and $A_{\oplus_{\pi\in \hat{G}\setminus\{\pi_{it}^+\}}\pi}$ intersect trivially.
Therefore by the Hahn Banach theorem,
there exists an element $S$ in $\VN_\omega(G)$ such that $\widetilde{\pi_{it}^+}(S)\neq 0$ and $\tilde{\pi}(S)=0$ for every other representation $\pi$
in $\hat{G}$. Hence by Equation (\ref{eq: uu'-integral-formula}),
\begin{equation*}
\langle uu', S\rangle=\int_{\hat{G}\times\hat{G}}\left[\int_{\hat{G}} \Tr(\pi''(S)T_{\pi,\pi';\pi''})dm_{\pi,\pi'}(\pi'')\right]d(\mu\times\mu')(\pi,\pi')=0,
\end{equation*}
where we used the fact that $m_{\pi,\pi'}$ is continuous on the principal continuous series. Therefore $S$ vanishes on $B_0(G)^2$ but does not vanish
on $A_{\pi_{it}^+}$. Moreover, it is known that $A_{\pi_{it}^+}$ is a subset of $B_0(G)$ (e.g. an easy consequence of Kunze-Stein phenomena).
Thus we conclude that $B_0(G)$ is not square-dense.
\end{proof}

The following corollary is a natural consequence of Theorem \ref{b0-SL-not-sq-dense} together with Cohen factorization theorem on Banach algebras with bounded
approximate identities.
\begin{corollary}
Let $G$ denote the group $\SL$. Then $B_0(G)$ does not have a bounded approximate identity, and is not (operator) amenable.
\end{corollary}
\begin{remark}
We recall that a (completely contractive) Banach
algebra is (operator) weakly amenable if every (completely) bounded derivation of the algebra into its dual space is inner (see \cite{Bade-Curtis-Dales} for the
Banach algebra case and \cite{Forrest-wood} for the operator space setting). Spronk \cite{NicoA(G)wa}, and independently Samei \cite{samei-opr-w.a-A(G)}, showed that the Fourier algebra of a locally compact group is always operator weakly amenable. Moreover, it easily follows from Theorem \ref{b0-SL-not-sq-dense} that $B_0(\SL)$ is not even (operator) weakly amenable. This contrasts with the above-mentioned results on the amenability of the Fourier algebras, and illustrates the inherent distinction between the Rajchman algebra and the Fourier algebra.
\end{remark}

\subsection{Discrete groups}

\begin{proposition}\label{prop: discrete-with-inf-abelian-subgroup}
Let $G$ be a discrete group which has an infinite Abelian subgroup
$H$. Then,
$B_0(G)$ is not operator amenable. In particular, for
any positive integer $n$, the free group $\mathbb{F}_n$ with $n$
generators is not operator amenable.
In addition, $B_0(G)$ does not have a bounded approximate identity.
\end{proposition}
\begin{proof}
Discrete groups are SIN-groups, and any subgroup of a
discrete group is closed. By Theorem~\ref{thm2}, the restriction map $r:B_0(G)\rightarrow B_0(H)$ is
a surjective completely contractive homomorphism. Assume that
$B_0(G)$ is operator amenable. Then by  $B_0(H)$ is operator
amenable as well, which contradicts Corollary \ref{cor:B_0(abelian-noncomp)Not-opr-amen}
since $H$ is non-compact.

Now assume by contradiction that $B_0(G)$ has a bounded approximate identity, and let
$\{u_\alpha\}$ be a bounded approximate identity of $B_0(G)$. Then $\{u_\alpha|_H\}$ is a bounded approximate identity for $B_0(H)$ which is
a contradiction with Theorem~\ref{thm:M0(Abelian)-no-b.a.i}.
\end{proof}

Let $G$ be a discrete group. The group $G$ is called \emph{periodic} if  every element of $G$ is of finite order.
The group $G$ is called \emph{locally finite} (respectively $\rm{F}_2$) if every finite (respectively two-element) subset of
$G$ generates a finite subgroup of $G$.
Clearly the class of locally finite groups is
contained in the class of $\rm{F}_2$ groups, which in turn is
contained in the class of periodic groups. It has been shown in
\cite{hall-kulatilaka} that every infinite locally finite group
contains an infinite Abelian subgroup. More generally, every
infinite $\rm{F}_2$ group contains an infinite Abelian subgroup (see
\cite{strunkov}). The following corollary follows easily from  Proposition \ref{prop: discrete-with-inf-abelian-subgroup}.

\begin{corollary}\label{cor:discrete-opr-amen=>periodic-loc.fin-F2}
Let $G$ be a discrete group such that $B_0(G)$ is operator amenable. Then
\begin{itemize}
\item[1.] $G$ is periodic.
\item[2.] If $G$ is locally finite, then $G$ is finite.
\item[3.] If $G$ is $\rm{F}_2$, then $G$ is finite.
 \end{itemize}
\end{corollary}
We can now characterize the discrete solvable groups whose Rajchman algebras are operator amenable.
\begin{theorem}\label{thm:solvable}
Let $G$ be a solvable discrete group such that $B_0(G)$ is operator amenable. Then $G$ is finite.
\end{theorem}
\begin{proof}
Suppose $G$ is solvable, i.e. it has a series $\{e\}=G_0\lhd G_1\lhd \ldots \lhd G_k=G$ such that $G_i$ is normal in $G_{i+1}$ and the quotient $G_{i+1}/G_i$
is Abelian for $i=0,\ldots,k-1$.
we proceed by induction on the length of the subnormal series:

\noindent Case 1:  If $k=1$, then $G$ is Abelian and we are done. So we start with $k=2$,
and assume that $\{e\}=G_0\lhd G_1\lhd G_2=G$ is a subnormal series such that $G_1$ and $G/G_1$ are Abelian.
By functorial properties of $B_0$, we have that $B_0(G_1)$ is operator amenable as well. Hence $G_1$ is finite by Corollary \ref{cor:B_0(abelian-noncomp)Not-opr-amen}.
 Now let $g_1,g_2$ be two elements in the group $G$, and let
$w=g_1^{\alpha_1}g_2^{\beta_1}\ldots g_1^{\alpha_n}g_2^{\beta_n}$ be a word in the group generated by $g_1$ and $g_2$. Then
\begin{eqnarray*}
g_1^{\alpha_1}g_2^{\beta_1}\ldots g_1^{\alpha_n}g_2^{\beta_n}G_1&=&(g_1^{\alpha_1}G_1)(g_2^{\beta_1}G_1)\ldots
(g_1^{\alpha_n}G_1)(g_2^{\beta_n}G_1)\\&=&(g_1^{\alpha_1}G_1)\ldots(g_1^{\alpha_n}G_1)\times
(g_2^{\beta_1}G_1)\ldots(g_2^{\beta_n}G_1)\\&=&g_1^{\sum \alpha_i}g_2^{\sum \beta_i}G_1,
\end{eqnarray*}
therefore every word in $\langle g_1,g_2\rangle$ is of the form $g_1^\alpha g_2^\beta z$ for some $z$ in $G_1$. Moreover $g_1$ and $g_2$ are periodic since the group has an operator amenable Rajchman algebra. Therefore $\langle g_1,g_2\rangle$ is finite, i.e. $G$ is $F_2$.
Recall that infinite  $F_2$ groups always have infinite Abelian subgroups, hence their Rajchman algebras cannot be operator amenable.
Therefore $G$ is finite.\\

\noindent Case 2: First note that the group is periodic. Suppose that for periodic solvable groups of subnormal series of length less than
 $n$, if $B_0(G)$ is operator amenable then $G$ is finite (induction hypothesis). Let $G$ be a periodic solvable group with the subnormal series
$\{1\}=G_0\leq G_1\leq \ldots \leq G_n=G$. Then by functorial properties and induction hypothesis, $G_{n-1}$ is finite. Repeating the same argument as in Case 1, we get that $G$ is finite as well.
\end{proof}

\bibliographystyle{plain}
\bibliography{thesis}

\end{document}